\documentclass[runningheads,12pt,a4paper]{article}
\usepackage[T1]{fontenc}

\usepackage{graphicx}
\usepackage{amsthm}
\usepackage{amsmath,amssymb}
\usepackage{enumitem}
\usepackage{xcolor}
\newtheorem{theorem}{Theorem}
\newtheorem{definition}{Definition}
\newtheorem{lemma}{Lemma}
\newtheorem*{sunlemma}{Sunflower Lemma}
\newtheorem*{sunconj}{Erd\H{o}s-Rado Sunflower Conjecture}

\newtheorem{corollary}{Corollary}

\newtheorem*{statement}{Statement of Theorem 1}

\newtheorem*{statement1}{Statement of Lemma 1}

\newtheorem*{statement5}{Statement of Lemma 5}
\begin{document}
	\title{Erd\H{o}s Rado Sunflower Theorem for Shifted Families}
	%
	%
	\author{Tapas Kumar Mishra \footnote{Orcid: 0000-0002-9825-3828} \\
		%
		%
		Department of Computer Science and Engineering,\\
		National Institute of Technology, Rourkela
		769008, India\\
		tap1cse@gmail.com, mishrat@nitrkl.ac.in\\
		https://mishra-tapas.github.io/}
	\maketitle              
	\begin{abstract}
		Let $f(k,s)$ denote the minimum integer $m$ such that
		any family $\mathcal{F}$ consisting of $k$-sized sets of cardinality at least $m$ always contain a sunflower of size $s$. The Erd\H{o}s-Rado Sunflower Conjecture states that 
		for every $s >2$, there is an constant $C=C(s)$ such that $f(k,s) \leq C^k$.
		In this paper, we prove the conjecture for shifted families.
		
		keywords: {Erd\H{o}s-Rado Sunflower Conjecture, Sunflower, Sunflower Lemma}
	\end{abstract}
	\section{Introduction}
	Let $n$ be any positive integer and 
	$[n]$ denote the set $\{1,\ldots,n\}$. Let $2^{[n]}$ denote the power set of $[n]$, $\binom{[n]}{k}$ denote the family of all the $k$-sized subsets of $[n]$.
	Let $\mathcal{F}$ be a family of subsets of $[n]$. 
	A sunflower (or $\Delta$ system) of size $s$ with \emph{core} (or \emph{kernel}) $Y$ is a collection $A_1,\ldots,A_s$ such that $A_i \cap A_j=Y$ for each $i\neq j$, the sets $A_i \setminus Y$ are the petals each of which is non-empty.
	Let $f(k,s)$ denote the minimum integer $m$ such that
	any family $\mathcal{F}$ consisting of $k$-sized sets of cardinality more than $m$ always contain a sunflower of size $s$. We have the following bound on $f(k,s)$ due to
	Erd\H{o}s and Rado \cite{erdosrado1960}.
	
	\begin{sunlemma}\label{sunlemma}
		$f(k,s) \leq k!(s-1)^k$.
	\end{sunlemma}
	
	Moreover Erd\H{o}s and Rado \cite{erdosrado1960} also proved a lower bound of $(s-1)^k$ on $f(k,s)$. They also proposed the following conjecture.
	
	\begin{sunconj}\label{sunconj}
		For every $s >2$, there is an constant $C=C(s)$ such that
		$f(k,s) \leq C^k$.
	\end{sunconj}
	This famous conjecture in extremal combinatorics
	was one of Erd\H{o}s’ favorite problems \cite{erdos1981}, for which he offered a \$1000 reward \cite{erdos1990}; it remains open despite
	considerable attention \cite{kostochka2000}.
	In 2019, Alweiss, Lovett, Wu and Zhang \cite{alweiss2020}, using iterative encoding arguments, proved that $f(k,s) 
	\leq (Cs^3\log k \log\log k)^k$ for some constant $C>0$.
	Rao \cite{rao2020} subsequently improved the bound to $f(k,s) \leq (Cs\log (ks))^k$. The current best known Sunflower bound
	due to Bell, Chueluecha, and Warnke \cite{Bell2021} is $(Cs\log k)^k$.

	A matching $\mathcal{M}$ is a collection of pairwise disjoint members of $\mathcal{F}$. The matching number,
	$\nu(\mathcal{F})$ is the cardinality of any maximum matching in $\mathcal{F}$. Note that any matching $\mathcal{M}$ is a sunflower with empty core.
	
	The \emph{sunflower number},
	$\tau(\mathcal{F})$ is the cardinality of any maximum sunflower in $\mathcal{F}$. Note that 
	$\nu(\mathcal{F}) \leq \tau(\mathcal{F})$.
	
	\subsection*{Our results}
	Firstly, we connect the sunflower number $\tau(\mathcal{F})$ and size of the family $\mathcal{F}$, and size of the shadow of the family $\partial \mathcal{F}$ in Section \ref{sec:shadow}: in particular, we derive a lower bound on the size of the shadow. This result is again of independent interest and gives tighter bound than the Kruskal-Katona Theorem provided $\tau(\mathcal{F}) \leq k$, for certain ranges of $k$. We have the following result in this direction.
	\begin{lemma} \label{lemma:4}
		$|\mathcal F| \leq \frac{\tau(\mathcal F)}{k}*|\partial \mathcal F|$.
	\end{lemma} 
	Next, we study the Frankl's $(i,j)$ shift/compression
	$\mathcal{C}_{ij}$ and study the variation of sunflower number under shift operations. Let $\mathcal{F}_{1,1}=\mathcal{F}$ and $\mathcal{F}_{i,j}=\mathcal{C}_{ij}(\mathcal{F}_{i,j-1})$, for $1 \leq i< j \leq n$.  Note that $\mathcal{F}_{n-1,n}$ is compressed: for any $F \in \mathcal{F}_{n-1,n}$, if $i \notin F$ and $j \in F$ and $i<j$, then $(F \setminus \{j\}) \cup \{i\} \in \mathcal{F}_{n-1,n}$. It is not hard to see that $\mathcal{F}_{n-1,n}$ is also stable: i.e. $\mathcal{C}_{ij}(\mathcal{F}_{n-1,n})=\mathcal{F}_{n-1,n}$, for $1 \leq i< j \leq n$ (see \cite{frankl1987} for a proof). We define the operator $\mathcal S$ as a sequence of these $\binom{n}{2}$ shifts on some family  $\mathcal{F}$ such that $\mathcal{S}(\mathcal{F})=\mathcal{F}_{n-1,n}$.
	We call a family \emph{shifted} if  $\mathcal{S}(\mathcal{F})=\mathcal{F}$. In Section \ref{sec:shift}, we provide an upper bound for
	the sunflower number of the shifted family, $\tau(\mathcal{S}(\mathcal{F}))$, in terms of the sunflower number of the original family, $\tau(\mathcal{F})$. 

	In Section \ref{sec:ss}, we use properties of shifted families combined with Lemma \ref{lemma:4} to find a large star hinged at elements in the ground set.
	\begin{lemma} \label{lemma:5}
		Let $\mathcal{F} \subseteq \binom{n}{k}$ be shifted and  let $\tau(\mathcal{F}) \leq k$.
		Let $\mathcal{B}_0=\mathcal{F}$.
		Let $\mathcal{A}_i=\{F \in \mathcal{B}_{i-1} : i \in F\}$,
		$\mathcal{B}_i= \{F \in \mathcal{B}_{i-1} : i \notin F\}$, for $1 \leq i \leq n-1$. Then,
		$|\mathcal{A}_i| \geq \frac{|\mathcal{B}_{i-1}|}{2}$.
	\end{lemma}
	
	Let $f'(k,s)$ denote the minimum integer $m$ such that
	any \emph{shifted} family $\mathcal{F}$ consisting of $k$-sized sets of cardinality more than $m$ always contain a sunflower of size $s$. We have the following recurrence for $f'(k,s)$.
	
	\begin{theorem}\label{thm:shifted}
		Let $s,k$ be  positive integers, $s \geq 1$, $k \geq 1$. Then, $$f'(k,s) \leq \begin{cases}
			s^{2k}, \text{ if $k\le s-1$,}\\
			2f'(k-1,s), \text{ Otherwise.}
		\end{cases}$$
	\end{theorem}
	
	This establishes Conjecture \ref{sunconj} for shifted families.
	
	
	\section{Sunflower Number and Shadows}
	\label{sec:shadow}
	\begin{definition}[Shadow]. Given a $k$-uniform set family $F$, the shadow $\partial F$ is defined as
		
		\[\partial \mathcal F = \{E : E = F \setminus \{x\} \text{ for some } F \in \mathcal F \text{ and } x \in F\}.\]
		
	\end{definition}
	\begin{statement1}
		$|\mathcal F| \leq \frac{\tau(\mathcal F)}{k}*|\partial \mathcal F|$.
	\end{statement1} 
	
	\begin{proof}
		We prove the inequality by double-counting the number of pairs $(E, F)$ such that $F \in \mathcal{F}$, $E \in \partial \mathcal{F}$, and $E \subset F$. Let this quantity be $N$.
		
		First, we calculate $N$ by summing over the sets in $\mathcal{F}$. Since $\mathcal{F}$ is $k$-uniform, every set $F \in \mathcal{F}$ contains exactly $\binom{k}{k-1} = k$ subsets of size $k-1$. By the definition of the shadow, all such subsets belong to $\partial \mathcal{F}$. Thus, each $F$ contributes exactly $k$ to the sum:
		
		$$N = \sum_{F \in \mathcal{F}} k = k \cdot |\mathcal{F}| \quad (*)$$
		
		Next, we calculate $N$ by summing over the sets in the shadow $\partial \mathcal{F}$. For each $E \in \partial \mathcal{F}$, let $d_{\mathcal{F}}(E)$ denote the number of sets in $\mathcal{F}$ that contain $E$ (often called the codegree or extension degree of $E$).$$N = \sum_{E \in \partial \mathcal{F}} d_{\mathcal{F}}(E)$$Consider the family of extensions of a specific set $E \in \partial \mathcal{F}$:$$\mathcal{F}_E = \{F \in \mathcal{F} : E \subset F\}$$Since any distinct $F_i, F_j \in \mathcal{F}_E$ share exactly the set $E$ (as $|F| = k$ and $|E| = k-1$, they can only differ by one element), the family $\mathcal{F}_E$ forms a sunflower with kernel $E$ and petals of size 1. The size of this sunflower is exactly $|\mathcal{F}_E| = d_{\mathcal{F}}(E)$.By definition, $\tau(\mathcal{F})$ is the size of the largest sunflower in $\mathcal{F}$. Therefore, for every $E \in \partial \mathcal{F}$, we must have:$$d_{\mathcal{F}}(E) \leq \tau(\mathcal{F})$$Substituting this bound into our summation for $N$:$$N = \sum_{E \in \partial \mathcal{F}} d_{\mathcal{F}}(E) \leq \sum_{E \in \partial \mathcal{F}} \tau(\mathcal{F}) = \tau(\mathcal{F}) \cdot |\partial \mathcal{F}| \quad (**)$$Combining $(*)$ and $(**)$:$$k \cdot |\mathcal{F}| \leq \tau(\mathcal{F}) \cdot |\partial \mathcal{F}|$$
	\end{proof}
	
	\section{Sunflower Number and 1-Compressed families }
	\label{sec:shift}
	\begin{definition}[$(i,j)$ shift/compression](\cite{frankl1987})
		Let $\mathcal{F} \subseteq 2^{[n]}$. For $1 \leq i,j \leq n$, define $\mathcal{C}_{ij}(\mathcal{F})= \{\mathcal{C}_{ij}(F): F \in \mathcal{F}\}$, where
		\begin{align*}
			\mathcal{C}_{ij}(F) = \begin{cases}
				(F \setminus \{j\})\cup \{i\} , \text{ if  $i \not \in F$, $j \in F$, and $(F \setminus \{j\})\cup \{i\} \not\in \mathcal{F}$},\\
				F, \text{ otherwise.}
			\end{cases}
		\end{align*}
	\end{definition}
	We have two lemmas on the effects of a $(i,j)$ shift on a family $\mathcal{F}$.
	
	\begin{lemma} \label{lemma:1}
		Let $\mathcal{F}$ be a family of subsets of $[n]$ and let $1 \leq i, j \leq n$. Then the following hold:
		\begin{enumerate}[label=(\roman*)]
			\item For any $F \in \mathcal{F}$, $|\mathcal{C}_{ij}(F)| = |F|$.
			\item $|\mathcal{C}_{ij}(\mathcal{F})| = |\mathcal{F}|$.
			\item The matching number does not increase, i.e., $\nu(\mathcal{C}_{ij}(\mathcal{F})) \leq \nu(\mathcal{F})$.
		\end{enumerate}
	\end{lemma}
	
	\begin{proof}
		The proof of (i) and (ii) are given in \cite{frankl1987}. As (iii) is not used in this paper and is of independent interest, the proof is given in the Appendix \ref{app:1}. 
	\end{proof}
	\begin{lemma} \label{lemma:2}
		Let $\mathcal{F}$ be a family of subsets of $[n]$ and let $1 \leq i, j \leq n$. Then  $\tau(\mathcal{C}_{ij}(\mathcal{F})) \leq 2\tau(\mathcal{F})$.
	\end{lemma}
	
	\begin{proof}
		As Lemma \ref{lemma:2} is not used directly in this paper and is of independent interest, the proof is given in the Appendix \ref{app:2}. 
	\end{proof}

	\subsection*{Compressed/Shifted/Initial families}	
	
	Let $\mathcal{F}_{1,1}=\mathcal{F}$ and $\mathcal{F}_{i,j}=\mathcal{C}_{ij}(\mathcal{F}_{i,j-1})$, for $1 \leq i< j \leq n$.  Note that $\mathcal{F}_{n-1,n}$ is compressed: for any $F \in \mathcal{F}_{n-1,n}$, if $i \notin F$ and $j \in F$ and $i<j$, then $(F \setminus \{j\}) \cup \{i\} \in \mathcal{F}_{n-1,n}$. It is not hard to see that $\mathcal{F}_{n-1,n}$ is also stable: i.e. $\mathcal{C}_{ij}(\mathcal{F}_{n-1,n})=\mathcal{F}_{n-1,n}$, for $1 \leq i< j \leq n$ (see \cite{frankl1987} for a proof). We define the operator $\mathcal S$ as a sequence of these $\binom{n}{2}$ shifts on some family  $\mathcal{F}$ such that $\mathcal{S}(\mathcal{F})=\mathcal{F}_{n-1,n}$.
	We call a family \emph{shifted} if  $\mathcal{S}(\mathcal{F})=\mathcal{F}$.

\section{Sunflower Number, Compressed families and Shadows}\label{sec:ss}

\begin{statement5}
	Let $\mathcal{F} \subseteq \binom{n}{k}$ be shifted and  let $\tau(\mathcal{F}) \leq k$.
	Let $\mathcal{B}_0=\mathcal{F}$.
	Let $\mathcal{A}_i=\{F \in \mathcal{B}_{i-1} : i \in F\}$,
	$\mathcal{B}_i= \{F \in \mathcal{B}_{i-1} : i \notin F\}$, for $1 \leq i \leq n-1$. Then,
	$|\mathcal{A}_i| \geq \frac{|\mathcal{B}_{i-1}|}{2}$.
\end{statement5}

\begin{proof}
	Note that for any  $F \in \mathcal{B}_{i-1}$, $F \cap [i-1]=0$.		
	Further, since $\mathcal{F}$ is shifted,  $\mathcal{B}_{i-1}$ is $i$-compressed (or $i$-stable): for any $F \in \mathcal{B}_{i-1}$, if $i \notin F$ and $j \in F$, then $(F \setminus \{j\}) \cup \{i\} \in \mathcal{B}_{i-1}$.
	From $i$-compression, it follows that 
	$$\{ S \cup \{i\} : S \in \partial \mathcal{B}_{i} \} \subseteq \mathcal{A}_i.$$
	So, $|\mathcal{A}_i| \geq  |\partial \mathcal{B}_{i}|.$
	From Lemma \ref{lemma:4}, $|\partial\mathcal{B}_i|\geq \frac{k}{\tau(\mathcal{B}_i)}|\mathcal{B}_i|$. 
	From the premise of the lemma, we have $\tau(\mathcal{B}_i) \leq \tau(\mathcal{F}) \leq k$.
	Therefore, $|\mathcal{A}_i|\geq |\mathcal{B}_i|$. Since $|\mathcal{B}_{i-1}|=|\mathcal{A}_i|+|\mathcal{B}_{i}|$, the lemma follows.
\end{proof}

\section{Sunflower Theorem for small $k$}
\begin{statement}
	Let $s,k$ be  positive integers,  $k \geq 1$, $s \geq k$. Then, $f(k,s) \leq s^{2k}$.
\end{statement}
\begin{proof}
	We proceed with induction on $k$.
	When $k=1$, the theorem holds. Assume the theorem holds for families, where each set is of size $k-1$. That is, any family of $(k-1)$-sized sets with cardinality more than $s^{2k-2}$ contains a sunflower of size $s$. 
	
	Let $\mathcal{F}$ be a family of $k$-sized sets on $[n]$ with $|\mathcal{F}| > s^{2k}$. We need to prove the existence of a sunflower of size more than $s$ in $\mathcal{F}$. 
	
	\noindent \textbf{Consider a maximal matching:}
	Let $\mathcal{M} = \{A_1, \dots, A_m\}$ be a maximal matching (a collection of pairwise disjoint sets) contained in $\mathcal{F}$.
	If $m > s$, then $\mathcal{M}$ itself is a sunflower of size $s$ with empty kernel ($Y = \emptyset$), and we are done. Assume for contradiction that $m \le s$.
	
	\noindent \textbf{Define the intersection set $U$:}
	Since $\mathcal{M}$ is maximal, every other set in $\mathcal{F}$ must intersect at least one set in $\mathcal{M}$. Let $U$ be the union of all sets in the matching:$$U = \bigcup_{i=1}^m A_i.$$
	The size of $U$ is bounded by:$|U| = m \cdot k \le sk.$
	
	\noindent \textbf{Apply the Pigeonhole Principle:} 
	Every set $S \in \mathcal{F}$ intersects $U$. By the Pigeonhole Principle, there must be at least one element $x \in U$ that is contained in a large number of sets from $\mathcal{F}$. Let $\mathcal{F}_x = \{ S \setminus \{x\} : S \in \mathcal{F}, x \in S \}$ be the ``link'' of $x$. This is a family of sets of size $k-1$. The size of $\mathcal{F}_x$ satisfies:$$|\mathcal{F}_x| \ge \frac{|\mathcal{F}|}{|U|} > \frac{s^{2k}}{ks}= \frac{s^2}{ks} s^{2k-2}\geq s^{2k-2}.$$
	By the induction hypothesis, $\mathcal{F}_x$ contains a sunflower $\mathcal{S} = \{B_1, \dots, B_{s+1}\}$ of size $s+1$ with some kernel $K$.
	
	\noindent \textbf{Construct the Sunflower in $\mathcal{F}$:}
	Now, add the element $x$ back to each set in $\mathcal{S}$. Let $A_i = B_i \cup \{x\}$. The sets $A_1, \dots, A_{s+1}$ are in $\mathcal{F}$. Their intersection is $(K \cup \{x\})$. Since $B_i \setminus K$ were disjoint non-empty leafs, $A_i \setminus (K \cup \{x\})$ are the same disjoint non-empty leafs. Thus, $A_1, \dots, A_{s+1}$ form a sunflower of size $s+1$ in $\mathcal{F}$.
	
\end{proof}

\section{Sunflower Theorem for shifted families}

\begin{statement}
	Let $s,k$ be  positive integers, $s \geq 1$, $k \geq 1$. Then, $$f'(k,s) \leq \begin{cases}
		s^{2k}, \text{ if $k\le s-1$,}\\
		2f'(k-1,s), \text{ Otherwise.}
	\end{cases}$$
\end{statement}

\begin{proof}
	The first half of the recurrence follows from the previous theorem.
	Let $\mathcal{F}$ be a shifted family of $k$-sized sets on $[n]$ with $k\geq s$ and $|\mathcal{F}| > 2f'(k-1,s)$. We need to prove the existence of a sunflower of size more than $s$ in $\mathcal{F}$. From Lemma \ref{lemma:5},
	$\mathcal{A}_1=\{F \in \mathcal{B}_{0}=\mathcal{F} : 1 \in F\}$ has cardinality at least half of $\mathcal{F}$:
	$\mathcal{A}_1 \geq \frac{|\mathcal{F}|}{2}>f'(k-1,s)$.
	Let $\mathcal{G}_1 = \{ F \setminus \{1\} : F \in \mathcal{A}_1\}$ be the ``link'' of $1$. This is a family of sets of size $k-1$. The size of $\mathcal{G}_1$ satisfies:$$|\mathcal{G}_1| =|\mathcal{A}_1| >f'(k-1,s).$$ 
	Therefore, $\mathcal{G}_1$ contains a sunflower $\mathcal{S} = \{B_1, \dots, B_{s}\}$ with some kernel $K$.
	Now, add the element $1$ back to each set in $\mathcal{S}$. Let $A_i = B_i \cup \{1\}$. Thus, $A_1, \dots, A_{s}$ form a sunflower of size $s$ in $\mathcal{F}$.
\end{proof}

We have the following corollary to the Theorem which establishes Conjecture \ref{sunconj} for shifted families.
\begin{corollary}\label{cor:1}
	$f'(k,s) \leq s^{2s-2}2^{k}.$
\end{corollary}	

%
%

\section{Discussion}

Resolution of the Sunflower Conjecture leads to resolution of various other problems in combinatorics.

\subsubsection*{The Weak (Erdős-Szemerédi) Sunflower Conjecture:}Let $\mathcal{F} \subseteq 2^{[n]}$ be a family of subsets of $[n]$ (of unrestricted, varying sizes). For every integer $k \geq 3$, there exists a constant $\gamma_k < 2$ (or equivalently, $\epsilon_k > 0$ such that $\gamma_k = 2^{1-\epsilon_k}$) such that if $\mathcal{F}$ is $k$-sunflower-free, then:$$|\mathcal{F}| \leq \gamma_k^n = 2^{(1-\epsilon_k)n}.$$
Here, the bounding parameter is $n$, the size of the ground set, rather than the size of the individual sets. 
It is a non-trivial theorem, established by Alon, Shpilka, and Umans \cite{alon2012}, that the Classical Sunflower Conjecture strictly implies the Weak Sunflower Conjecture.

\subsubsection*{Monotone Circuits \cite{rao2023}:}

In monotone circuit complexity, Razborov’s approximation method proves lower bounds for the $k$-CLIQUE problem by replacing exact OR gates with bounded Disjunctive Normal Forms. To prevent exponential blowup, large implicant families are "plucked" down to their cores using the sunflower lemma, introducing a controlled false-positive error. The classical $(k-1)^s s!$ bound severely restricts this process, yielding a suboptimal $n^{\Omega(\sqrt{k})}$ circuit lower bound. If the strong Erdős-Rado Sunflower Conjecture ($|\mathcal{F}| \leq c_k^s$) holds, this factorial bottleneck vanishes. The relaxed constraints allow plucking with drastically reduced errors, instantly pushing the monotone circuit lower bound to its absolute theoretical limit of $n^{\Omega(k)}$.

\subsubsection*{ Algorithmic Matrix Multiplication \cite{rao2023}:}

The sunflower conjecture imposes severe constraints on algebraic approaches to fast matrix multiplication, specifically concerning the matrix multiplication exponent $\omega$.
\begin{itemize}
\item The Coppersmith-Winograd Framework: The group-theoretic approach to bounding $\omega$ (introduced by Cohn and Umans) relies on finding large subsets of groups that avoid specific arithmetic structures, such as uniquely solvable puzzles.
\item Refutation of Abelian Approaches: Alon, Shpilka, and Umans established that if the Erdős-Rado sunflower conjecture is true, it implies a negative answer to the "no three disjoint equivoluminous subsets" question by Coppersmith and Winograd. Consequently, proving the conjecture would mathematically rule out the possibility of proving $\omega = 2$ using abelian group algebras, forcing the pursuit of optimal matrix multiplication algorithms entirely into the domain of non-abelian groups. 
\end{itemize}

\subsubsection*{ Data Structure Lower Bounds and Information Theory \cite{rao2023}:}

In the context of static data structures, the sunflower lemma governs the memory representation of overlapping queries.
\begin{itemize}
\item Memory Access Patterns: If a data structure stores subsets using adaptive cell probing, queries that share a large intersection (forming a sunflower) force the structure to encode distinct subset responses strictly within the shared memory cells of the sunflower's core.

\item Tightening Time-Space Tradeoffs: A fully resolved sunflower conjecture yields optimal lower bounds for the space complexity of answering intersection and minimum-element queries. It proves that any sufficiently large family of queried access sets must structurally overlap in a way that creates a strict, quantifiable bottleneck in memory cell representation.
\end{itemize}
\bibliographystyle{alpha}
\bibliography{SUN}
\appendix
\section{Matching number after a shift}\label{app:1}

\begin{lemma}
	Let $\mathcal{F}$ be a family of subsets of $[n]$ and let $1 \leq i, j \leq n$. Then the following hold:
	\begin{enumerate}[label=(\roman*)]
		\item For any $F \in \mathcal{F}$, $|\mathcal{C}_{ij}(F)| = |F|$.
		\item $|\mathcal{C}_{ij}(\mathcal{F})| = |\mathcal{F}|$.
		\item The matching number does not increase, i.e., $\nu(\mathcal{C}_{ij}(\mathcal{F})) \leq \nu(\mathcal{F})$.
	\end{enumerate}
\end{lemma}
\begin{proof}
	(i) The operator $\mathcal{C}_{ij}$ either leaves a set $F$ unchanged or replaces one element, $j$, with another, $i$. In both cases, the cardinality of the set is preserved.
	
	(ii) The map $F \mapsto \mathcal{C}_{ij}(F)$ is a bijection from $\mathcal{F}$ to $\mathcal{C}_{ij}(\mathcal{F})$. To see this, we only need to show it is injective. Suppose $\mathcal{C}_{ij}(F_1) = \mathcal{C}_{ij}(F_2) = G$. If neither set was shifted, $F_1 = F_2 = G$. If both were shifted, then $F_1 = (G \setminus \{i\}) \cup \{j\}$ and $F_2 = (G \setminus \{i\}) \cup \{j\}$, which implies $F_1 = F_2$. A case where one is shifted and the other is not (e.g., $G = (F_1 \setminus \{j\}) \cup \{i\}$ and $G=F_2$) leads to a contradiction with the definition of the shift, as the existence of $F_2 \in \mathcal{F}$ would have prevented $F_1$ from being shifted to $G$. Thus, the map is a bijection and $|\mathcal{C}_{ij}(\mathcal{F})| = |\mathcal{F}|$.
	
	Let $\mathcal{F}' = \mathcal{C}_{ij}(\mathcal{F})$. To show that $\nu(\mathcal{F}') \leq \nu(\mathcal{F})$, we will take a maximum matching in $\mathcal{F}'$ and construct a matching of the same size in $\mathcal{F}$.
	Let $\mathcal{M}' = \{B_1, B_2, \dots, B_p\}$ be a maximum matching in $\mathcal{F}'$, so $p = \nu(\mathcal{F}')$. The sets $B_k$ are pairwise disjoint. For each $B_k \in \mathcal{M}'$, there is a unique preimage $A_k \in \mathcal{F}$ such that $B_k = \mathcal{C}_{ij}(A_k)$.	
	A set $B_k$ is altered by the shift (i.e., $B_k \neq A_k$) only if $B_k = (A_k \setminus \{j\}) \cup \{i\}$. This implies that $i \in B_k$. Since the sets in the matching $\mathcal{M}'$ are pairwise disjoint, at most one set in $\mathcal{M}'$ can contain the element $i$. Therefore, \textbf{at most one set in $\mathcal{M}'$ could have been altered} by the $\mathcal{C}_{ij}$ operation.	
	We consider two cases:
	
	\textbf{Case 1: No set in $\mathcal{M}'$ was altered.}
	In this case, $B_k = A_k$ for all $k=1, \dots, p$. This means every set in $\mathcal{M}'$ is also in $\mathcal{F}$. Thus, $\mathcal{M}'$ itself is a matching of size $p$ in $\mathcal{F}$, so $\nu(\mathcal{F}) \ge p$.
	
	\textbf{Case 2: Exactly one set in $\mathcal{M}'$ was altered.}
	Without loss of generality, let $B_1$ be the altered set. Then:
	\begin{itemize}
		\item $B_1 = (A_1 \setminus \{j\}) \cup \{i\}$, where $A_1 \in \mathcal{F}$, $i \notin A_1$, $j \in A_1$, and $B_1 \notin \mathcal{F}$.
		\item For all $k \geq 2$, $B_k = A_k \in \mathcal{F}$, and since $B_1 \cap B_k = \emptyset$, we must have $i \notin B_k$.
	\end{itemize}
	
	Now, we construct a matching $\mathcal{M}$ in $\mathcal{F}$. Consider the collection $\{A_1, B_2,\allowbreak \dots, B_p\}$.
	If this collection is a matching, we are done. This occurs if $j \notin B_k$ for all $k \geq 2$, as $A_1$ is then disjoint from all $B_k$.
	Suppose the collection is not a matching. This can only happen if $A_1$ intersects some $B_k$ for $k \geq 2$. Since $A_1 = (B_1 \setminus \{i\}) \cup \{j\}$ and $B_1$ is disjoint from all other $B_k$, the intersection must be the element $j$. As the sets $\{B_k\}_{k \geq 2}$ are disjoint, at most one of them can contain $j$. Let this be $B_2$, so $j \in B_2$.
	Now we have $j \in B_2$ and $i \notin B_2$. Since $B_2$ was not altered by the shift (i.e., $\mathcal{C}_{ij}(B_2) = B_2$), the condition for shifting must have failed. This implies that the set $C_2 = (B_2 \setminus \{j\}) \cup \{i\}$ must already be in $\mathcal{F}$.
	
	Let us define a new collection of sets $\mathcal{M} = \{A_1, C_2, B_3, \dots, B_p\}$. We claim this is a matching of size $p$ in $\mathcal{F}$.
	\begin{enumerate}
		\item \textbf{Membership in $\mathcal{F}$}: $A_1 \in \mathcal{F}$ by definition. $C_2 \in \mathcal{F}$ as argued above. $B_k \in \mathcal{F}$ for $k \geq 3$ by assumption.
		\item \textbf{Pairwise Disjointness}:
		\begin{itemize}
			\item The sets $\{B_3, \dots, B_p\}$ are pairwise disjoint.
			\item For $k \geq 3$, $A_1 \cap B_k = \emptyset$ because $B_1 \cap B_k = \emptyset$ and $B_k$ contains neither $i$ nor $j$.
			\item For $k \geq 3$, $C_2 \cap B_k = \emptyset$ because $B_2 \cap B_k = \emptyset$ and $B_k$ contains neither $i$ nor $j$.
			\item Crucially, $A_1 \cap C_2 = \emptyset$. We know $A_1 = (B_1 \setminus \{i\}) \cup \{j\}$ and $C_2 = (B_2 \setminus \{j\}) \cup \{i\}$. Since $B_1$ and $B_2$ are disjoint, any intersection between $A_1$ and $C_2$ must involve $i$ or $j$. However, $i \notin A_1$ and $j \notin C_2$. Therefore, they are disjoint.
		\end{itemize}
	\end{enumerate}
	We have successfully constructed a matching $\mathcal{M}$ of size $p$ in $\mathcal{F}$.
	
	In all cases, the existence of a matching of size $p$ in $\mathcal{F}'$ implies the existence of a matching of size at least $p$ in $\mathcal{F}$. Therefore, $\nu(\mathcal{F}) \ge \nu(\mathcal{F}')$. 
\end{proof}

\section{Sunflower number  after a shift}\label{app:2}
\begin{lemma}
	Let $\mathcal{F}$ be a family of subsets of $[n]$ and let $1 \leq i, j \leq n$. Then  $\tau(\mathcal{C}_{ij}(\mathcal{F})) \leq 2\tau(\mathcal{F})$.
\end{lemma}

\begin{proof}
	Let $\mathcal{G} = \mathcal{C}_{ij}(\mathcal{F})$ be the shifted family. Let $\tau(\mathcal{G})$ denote the size of the maximum sunflower in the shifted family. Let $\mathcal{S} = \{S_1, S_2, \dots, S_m\} \subseteq \mathcal{G} $ be a sunflower of size $m = \tau(\mathcal{G} )$. Let $K$ be the kernel of the sunflower $\mathcal{S}$, i.e., $S_p \cap S_q = K$ for all $p \neq q$.
	We partition the sunflower $\mathcal{S}$ into two sub-families based on whether the sets were ``shifted'' or remained ``unchanged'' from the original family $\mathcal{F}$.
	
	{\noindent \bf Step 1: Partitioning the Sunflower}
	
	Recall the definition of the shift $\mathcal{C}_{ij}(F)$. A set $S_r \in \mathcal{G} $ is formed in one of two ways:
	\begin{itemize}
		\item Shifted: $S_r = (F_r \setminus \{j\}) \cup \{i\}$ where $F_r \in \mathcal{F}$, $j \in F_r, i \notin F_r$, and the potential shift was not already in $\mathcal{F}$. In this case, $i \in S_r$ and $j \notin S_r$.
		\item Unchanged: $S_r = F_r$ where $F_r \in \mathcal{F}$.
	\end{itemize}
	
	We define two index sets:$$I_{shifted} = \{ r \in [m] : S_r \text{ was shifted from some } F_r \in \mathcal{F} \}$$$$I_{unchanged} = \{ r \in [m] : S_r \text{ was not shifted, i.e., } S_r \in \mathcal{F} \}$$
	
	Let $\mathcal{A} = \{ S_r : r \in I_{shifted} \}$ and $\mathcal{B} = \{ S_r : r \in I_{unchanged} \}$.
	Clearly, $|\mathcal{S}| = |\mathcal{A}| + |\mathcal{B}|$.
	
	{\noindent \bf Step 2: Analyzing the Unchanged Part ($\mathcal{B}$)}
	
	The family $\mathcal{B}$ is a subset of the original sunflower $\mathcal{S}'$. Therefore, $\mathcal{B}$ itself is a sunflower (any subset of a sunflower is a sunflower). Furthermore, by definition of $I_{unchanged}$, every set in $\mathcal{B}$ belongs to the original family $\mathcal{F}$. Thus, $\mathcal{B}$ is a sunflower contained in $\mathcal{F}$. By the definition of the sunflower number $\tau(\mathcal{F})$, the size of this sunflower cannot exceed the maximum:$$|\mathcal{B}| \leq \tau(\mathcal{F})$$
	
	{\noindent \bf  Step 3: Analyzing the Shifted Part ($\mathcal{A}$)} 
	
	If $\mathcal{A}$ is empty or contains only 1 set, then $|\mathcal{A}| \le 1 \le \tau(\mathcal{F})$ (assuming the family is non-trivial), and we are done. Assume $|\mathcal{A}| \ge 2$.For every $S_r \in \mathcal{A}$, we know by the definition of the shift that $i \in S_r$ and $j \notin S_r$. Since every set in $\mathcal{A}$ contains $i$, the kernel of $\mathcal{A}$ (which is the intersection of all sets in $\mathcal{A}$) must contain $i$. Let $K_{\mathcal{A}}$ be the kernel of the sunflower $\mathcal{A}$. We have $i \in K_{\mathcal{A}}$ and $j \notin K_{\mathcal{A}}$. Now consider the pre-images of these sets in $\mathcal{F}$. For each $S_r \in \mathcal{A}$, let $F_r$ be the original set in $\mathcal{F}$. Since $S_r$ was shifted, we have the relationship:$$S_r = (F_r \setminus \{j\}) \cup \{i\}$$
	Inverting this, the original set is:$$F_r = (S_r \setminus \{i\}) \cup \{j\}$$
	Let $\mathcal{F}_{\mathcal{A}} = \{ F_r : r \in I_{shifted} \} \subseteq \mathcal{F}$. We claim that $\mathcal{F}_{\mathcal{A}}$ forms a sunflower in $\mathcal{F}$. Take any two distinct sets $F_p, F_q \in \mathcal{F}_{\mathcal{A}}$. Their intersection is:\begin{align*}F_p \cap F_q &= \left( (S_p \setminus {i}) \cup {j} \right) \cap \left( (S_q \setminus {i}) \cup {j} \right) &= \left( (S_p \cap S_q) \setminus {i} \right) \cup {j}
	\end{align*}
	Since $S_p$ and $S_q$ are from the sunflower $\mathcal{A}$, their intersection is $K_{\mathcal{A}}$.
	Thus, $F_p \cap F_q = (K_{\mathcal{A}} \setminus {i}) \cup {j}$. This intersection depends only on $K_{\mathcal{A}}$, $i$, and $j$, and is identical for all pairs in $\mathcal{F}_{\mathcal{A}}$. Therefore, $\mathcal{F}_{\mathcal{A}}$ is a sunflower in $\mathcal{F}$. The size of this sunflower is $|\mathcal{F}_{\mathcal{A}}| = |\mathcal{A}|$. By the definition of $\tau(\mathcal{F})$:$$|\mathcal{A}| \leq \tau(\mathcal{F})$$
	
	Combining the bounds from Step 2 and Step 3:\[\tau(\mathcal{G}) = m = |\mathcal{A}| + |\mathcal{B}| \leq \tau(\mathcal{F}) + \tau(\mathcal{F})= 2\tau(\mathcal{F}).\]
\end{proof}

\end{document}